\documentclass[a4paper,12pt]{article}
\usepackage[cp1250]{inputenc}
\usepackage[T1]{fontenc}
\usepackage{amsfonts}
\usepackage{enumerate}
\usepackage{color}
\usepackage{amsmath}
\usepackage{amssymb}
\usepackage{amsthm} 
\usepackage{textcomp}
\usepackage{wrapfig}
\usepackage{yfonts} 
\usepackage{url}
\usepackage{fancyhdr}
\usepackage{geometry}
\usepackage{stmaryrd}
\usepackage{bbding}
\usepackage{enumitem} 
\usepackage{bbm}
\everymath{\displaystyle}
\theoremstyle{plain} 
\newtheorem{tw}{Theorem}[section]	

\theoremstyle{definition} 
\newtheorem{definition}{Definition}

\newcommand\cA{{\mathcal A}}
\newcommand\cF{{\mathcal F}}
\newcommand\cS{{\mathcal S}}
\newcommand\cM{{\mathcal M}}
\newcommand\fS{{\textfrak{S}}}

\newcommand\bIs{\mathbf{I}_{\rS}}
\newcommand\bI{\mathbf{I}}

\newcommand{\rS}{\mathrm{S}}
\newcommand{\rM}{\mathrm{M}}

\renewcommand\ge{\geqslant}
\renewcommand\le{\leqslant}

\title{Convergence theorems for seminormed fuzzy integrals: Solutions to  Hutn\`ik's open problems}
\author{Michał Boczek
, Marek Kaluszka
\\ 
{\emph{
\small{Institute of Mathematics, Lodz University of Technology, 90-924 Lodz, Poland}}}}
\date{}
\begin{document}
\maketitle

\begin{abstract} 
In this note, we give solutions to Problems $9.4$ and $9.5,$ which were presented by Mesiar and Stup\v{n}anov\'{a} $\cite{mesiar11}$ and by 
Borzová-Molnárová,  Hal\u{c}inová and Hutník, in [{\it The smallest semicopula-based universal integrals I: properties and characterizations,} Fuzzy Sets and Systems (2014),  http://dx.doi.org/
10.1016/j.fss.2014.09.024]. 
\end{abstract}

{\it Keywords: }{Generalized Sugeno integral; Seminormed Sugeno integral; Capacity; Semicopula; Seminorm.}

\section{Introduction}
Let $(X,\cA)$ be a~measurable space, where $\cA$ is a~$\sigma$-algebra of subsets of a non-empty set $X,$ and let $\cS$ be the family of all measurable spaces. The~class of all $\cA$-measurable functions $f\colon X\to [0,1]$ is denoted by $\cF_{(X,\cA)}.$ A~{\it capacity} on $\cA$  is a~non-decreasing set function 
$\mu\colon \cA\to [0,1]$ with $\mu(\emptyset)=0$ and  $\mu(X)=1.$
We denote by $\cM_{(X,\cA)}$ the class of all 
capacities  on $\cA.$  

Suppose that $\rS\colon [0,1]^2\to [0,1]$ is a~semicopula (also called 
a~$t$-{\it seminorm}), i.e., a~non-decreasing function in both coordinates with the neutral element equal to $1.$  It is clear that $\rS(x,y)\le x\wedge y$ 
and $\rS(x,0)=0=\rS(0,x)$ for all $x,y\in [0,1],$ where $x\wedge y=\min(x,y)$ (see $\cite{bas,dur,klement2}$).  We denote the class of all semicopulas
by $\fS$. 
Typical examples of semicopulas are the functions: 
$\rM(a,b)=a\wedge b,$ $\Pi(a,b)=ab,$ $\rS(x,y)=xy(x\vee y)$ and $\rS _L(a,b)=(a+b-1)\vee 0.$ 
Hereafter, $a\wedge b=\min(a,b)$ and $a\vee b=\max(a,b)$.

A~generalized Sugeno integral is defined by
\begin{align*}
\bIs(\mu,f):=\sup_{t\in [0,1]} \rS\big(t,\mu\big(\lbrace f\ge t\rbrace \big)\big),
 \end{align*} 
where $\lbrace f\ge t\rbrace=\lbrace x\in X\colon f(x)\ge t\rbrace,$  $(X,\cA)\in \cS$ and $(\mu, f)\in \cM_{(X,\cA)}\times \cF_{(X,\cA)}.$ The functional $\bIs$ is also called {\it seminormed fuzzy integral} $\cite{suarez,klement3,ouyang3}.$
Replacing semicopula $\rS$ with $\rM$, we get the {\it Sugeno integral} $\cite{sugeno1}$. Moreover, if  $\rS=\Pi,$ then $\bI_{\Pi}$ is called the {\it Shilkret integral} $\cite{shilkret}.$

 \section{Main results}
We present solutions to Problems $9.4$ and $9.5,$  which were posed by Hutník $\cite{mesiar11}$ (see also  $\cite{hutnik2}$, problems 2.18-2.19). 

\begin{definition}[$\cite{hutnik2}$]
Let $(X,\cA)\in\cS,$  $\mu\in \cM_{(X,\cA)},$ $(f_n)_{n=1}^\infty \subset \cF_{(X,\cA)}$ and $f\in\cF_{(X,\cA)}.$
\begin{enumerate}
\item We say that $(f_n)_{n=1}^\infty$ {\it converges in  $\mu$} to $f$ if  $\lim_{n\to \infty} \mu\big(\lbrace |f_n-f|\ge t\rbrace \big)=0$ for every $t\in (0,1].$  We write this as $f_n\xrightarrow{\mu} f.$
\item A sequence $(f_n)_{n=1}^\infty $ {\it converges strictly in  $\mu$} to $f,$ ($f_n\xrightarrow{\text{s}-\mu} f$), if $\lim_{n\to \infty} \mu\big(\lbrace |f_n-f|> 0\rbrace \big)=0.$
\item We say that $(f_n)_{n=1}^\infty $ {\it converges in mean} to $f$ with respect to the integral $\bIs$, 
($f_n\xrightarrow{\bIs} f$), if  $\lim_{n\to \infty}\bIs\big(\mu,|f_n-f|\big)=0.$
\end{enumerate}
\end{definition}
\medskip

 {\bf Problem 9.4} {\it
Characterize all the capacities for which strict convergence in measure is equivalent to convergence in measure on any measurable space.
}

\begin{tw}\label{tw1}
If $f_n\xrightarrow{s-\mu} f,$ then  $f_n\xrightarrow{\mu} f$  
for all $(X,\cA)\in\cS$, all  $\mu\in \cM_{(X,\cA)}$ and all $f,f_n\in \cF_{(X,\cA)}.$ The reverse implication is not true. 
\end{tw}

\begin{proof} Since  
$\mu\big(\lbrace |f_n-f|\ge t\rbrace \big)\le \mu\big(\lbrace |f_n-f|>0\rbrace \big)$ for every $t>0$, the convergence $f_n\xrightarrow{s-\mu} f$
implies $f_n\xrightarrow{\mu} f.$ The reverse implication is false. Indeed,  
let $(X,\cA)\in\cS$ and $\mu\in \cM_{(X,\cA)}.$ Put $f_n(x)=a_n$ for $x\in X$, where $\lim_{n\to\infty} a_n=0$ and $a_n>0$ for all $n$.
Then  $f_n\xrightarrow{\mu} 0,$ but the sequence $(f_n)$ does not 
converge strictly in  $\mu$ to $f=0.$ 
\end{proof}

\medskip
{\bf Problem 9.5} {\it
For which class of semicopulas (of capacities, eventually) is strict convergence in measure equivalent to mean convergence?
}\medskip

\begin{tw} 
If   $f_n\xrightarrow{s-\mu} f$ then  $f_n\xrightarrow{\bIs} f$
for all $(X,\cA)\in\cS,$ all $\mu\in \cM_{(X,\cA)}$ and all $f,f_n\in \cF_{(X,\cA)}.$ The converse implication does not hold. 
\end{tw}

\begin{proof}
From Theorem \ref{tw1} it follows that $\mu\big(\lbrace |f_n-f|\ge t\rbrace \big)\to 0$ as $n\to \infty$ for every $t>0$.
The function $t\to \mu\big(\lbrace |f_n-f|\ge t\rbrace \big)$ is non-increasing, so 
for every $\varepsilon>0$ there exists $n$ such that for all $k\ge n$
$$
\sup _{0\le t\le 1}\big(t\wedge  \mu(\lbrace |f_k-f|\ge t\rbrace )\big)\le \varepsilon.
$$
Since $S(a,b)\le a\wedge b$, we get  $\bIs\big(\mu,|f_n-f|\big)\to 0$ as $n\to \infty.$

The implication in the opposite direction is not true. In fact, 
put $f_n(x)=a_n$ for all $x\in X$, where $\lim_{n\to\infty} a_n=0$ and $a_n>0$ for all $n.$ Observe that   
\begin{align*}
\lim_{n\to\infty} \bIs\big(\mu,f_n\big)&=\lim_{n\to\infty}\sup _{0\le t\le a_n} \rS\big(t,\mu (X)\big)\\
&=\lim_{n\to\infty}\rS\big(a_n,1\big)=\lim_{n\to\infty} a_n=0,
\end{align*}
so $f_n\xrightarrow{\bIs} 0,$ but $\lim_{n\to\infty} \mu\big(|f_n|>0)=1,$
which completes the proof.
\end{proof}

\end{document}